\documentclass[11pt, reqno]{amsart}
\usepackage{amsmath,amssymb,amsfonts,amscd,hyperref,color}
\usepackage[latin1]{inputenc}

\usepackage[abs]{overpic}
\usepackage{verbatim}

\newcommand{\Hmm}[1]{\leavevmode{\marginpar{\tiny%
$\hbox to 0mm{\hspace*{-0.5mm}$\leftarrow$\hss}%
\vcenter{\vrule depth 0.1mm height 0.1mm width \the\marginparwidth}%
\hbox to
0mm{\hss$\rightarrow$\hspace*{-0.5mm}}$\\\relax\raggedright #1}}}

\newtheorem{theorem}{Theorem}
\newtheorem{thm}{Theorem}[section]

\newtheorem{lemma}[thm]{Lemma}
\newtheorem{pro}[thm]{Proposition}

\theoremstyle{definition}

\newtheorem{rem}[thm]{Remark}

\numberwithin{equation}{section}

\newcommand{\R}{{\mathbb R}}

\newcommand{\N}{{\mathbb N}}

\newcommand{\ph}{{\varphi}}

\begin{document}

\title{An improved discrete Hardy inequality}
\markright{Improved Hardy inequality}
\author[M.~Keller]{Matthias Keller}
\address{M.~Keller,  Institut f\"ur Mathematik, Universit\"at Potsdam
\\14476  Potsdam, Germany}
\email{mkeller@math.uni-potsdam.de}
\author[Y.~Pinchover]{Yehuda Pinchover}
\author[F.~Pogorzelski]{Felix Pogorzelski}
\address{F.~Pogorzelski and Y.~Pinchover,  Department of Mathematics, Technion-Israel Institute of Technology, 32000 Haifa, Israel}
\email{pincho@technion.ac.il}
 \email{felixp@technion.ac.il}

\maketitle

\begin{abstract}
We improve the classical discrete Hardy inequality
\begin{equation*}\label{1}
\sum _{{n=1}}^{\infty }a_{n}^{2}\geq \left({\frac  {1}{2}}\right)^{2} \sum _{{n=1}}^{\infty }\left({\frac  {a_{1}+a_{2}+\cdots +a_{n}}{n}}\right)^{2},
\end{equation*}
where $\{a_n\}_{n=1}^\infty$ is any sequence of non-negative real numbers.
\end{abstract}

In 1921 Landau wrote a letter to Hardy including a proof of the inequality
\begin{equation}\label{11}
\sum _{{n=1}}^{\infty }a_{n}^{p}\geq \left(\frac{p-1}{p}\right)^{p} \sum _{{n=1}}^{\infty }\left({\frac  {a_{1}+a_{2}+\cdots +a_{n}}{n}}\right)^{p},
\end{equation}
where $\{a_n\}_{n=1}^\infty$ is any sequence of non-negative real numbers and $1<p<\infty$.
This inequality was stated before by Hardy, and therefore, it is called a \emph{Hardy
inequality} (see \cite{KMP} for a marvelous
description on the prehistory of the celebrated Hardy inequality). Since then Hardy type inequalities have received an
enormous amount of attention.

Let $C_c(\N)$ be the space of finitely supported functions on $\N\!:=\! \{1,2,3,\ldots\}$. Inequality \eqref{11} is clearly equivalent to the following inequality
\begin{align}\label{CHI}
 \sum_{n=1}^{\infty}|\ph(n)-\ph(n-1)|^{p}\ge
\left(\frac{p-1}{p}\right)^{p}\sum_{n=1}^{\infty}\frac{|\ph(n)|^{p}}{n^{p}}
\end{align}
for all $\ph\in C_c(\N)$ where we set $\ph(0)=0$ as a ``Dirichlet boundary condition''.

%One particular focus in the literature lies on finding the sharp
%constant $C$ to a prescribed weight $w$ which is typically
%an inverse square weight. For the
%classical literature we refer here to the monographs
%\cite{BEL15,OK90} and to references therein. Recent developments
%include relationships with other functional inequalities, Hardy-type inequalities related to different boundary conditions
%\cite{AY,DEFT,KL}, Hardy-type inequalities for fractional Schr\"odinger operators \cite{BD,FS08,FS10,LS}, and
%Hardy inequalities for the Laplacian on metric trees \cite{EFK08,NS}.

The goal of this note is to prove the following improvement
of the classical Hardy inequality \eqref {CHI} for the case $p=2$.

\begin{theorem} \label{thm:MAIN}
Let $\varphi\in C_c(\N)$ and $\varphi(0)=0$. Then,
\begin{equation}\label{2}
\sum_{n=1}^{\infty} \big( \varphi(n) -\varphi(n-1) \big)^2 \geq \sum_{n=1}^{\infty} w(n)\varphi(n)^2,
\end{equation}
where $w$ is a strictly positive function on $\mathbb{N}$ given by
\begin{eqnarray*}
w(n)=
         \sum_{k=1}^{\infty} \!\!\binom{4k}{2k} \frac{1}{(4k-1)2^{4k-1}} \frac{1}{n^{2k}}
= \frac{1}{4n^2} + \frac{5}{64n^4} + \frac{21}{512n^6} \!+\!\ldots,
  \end{eqnarray*}
for $ n\ge 2 $ and $ w(1)=
2-\sqrt{2} $.
In particular, $w(n)$ is strictly greater than the classical Hardy weight $w_H(n):=1/(2n)^2$, $ n\in\mathbb{N} $.
\end{theorem}
\begin{rem}
In \cite{KPP} we show that \eqref{2} cannot be improved and is optimal in a certain sense.
In particular, there is no function $\widetilde{w} \gneqq w$ such that~\eqref{2} holds with $\widetilde{w}$
instead of $w$.
\end{rem}
%\section{Proof of Theorem~\ref{thm:MAIN}}
A nonzero mapping $u:\N \to \R$ taking non-negative values is called
 {\em positive} and we write $u \geq 0$ in this case.
If $u(n) > 0$ for all $n \in \N$, we call
$u$ a {\em weight} and we set $ u(0)=0 $.
For  weights $u$ we write $\ell^2(\N, u)$ for
the space of all functions $f:\N \to \R$ such that $\sum_{n=1}^\infty f(n)^2 u(n)<\infty$.  Clearly, $\ell^2(\N, u)$ equipped with the scalar product
\begin{align*}
\langle{f},g\rangle_{u}:=\sum_{n=0}^\infty f(n)g(n)u(n)
\qquad f,g \in \ell^2(\N, u)
\end{align*}
is a Hilbert space with the induced norm $\|\cdot\|_{u}$.   We clearly have:
\begin{lemma} \label{lemma:unitary}
For any weight $u:\N \to (0,\infty)$, the operator
$T_u : \ell^2(\N,u^2) \to \ell^2(\N)$, given by $(T_u\varphi)(n):= u(n) \varphi(n)$
is unitary.
\end{lemma}
We consider $\mathbb{N}_0=\{0,1,2,\ldots\}$ as the standard graph on the non-negative integers. Note that here $ 0 $ is added as a ``boundary point'' of $ \mathbb{N} $.
We say that $n,m \in \N_0$ are
{\em connected by an edge} (and we write $n \sim m$) if $|n-m| = 1$.

Given a weight $u$,
the {\em combinatorial Laplacian} $\Delta_u$ associated with $u$
is given by
\[
\Delta_{u} \varphi(n) := \frac{1}{u(n)^2} \sum_{m \sim n} u(n)u(m)\big( \varphi(n) - \varphi(m) \big) \qquad n\ge1,
\]
where $\varphi$ is an arbitrary function on $\N$ and we set $ \varphi(0)=0 $ (as we set $ u(0)=0 $ above). This can be understood as a ``Dirichlet boundary condition'' at $ n=0 $.
The operator $\Delta := \Delta_{1}$ is the standard discrete combinatorial
Laplace operator corresponding to the constant weight $u=1$. % as defined above.

\begin{lemma}[Ground state transform] \label{lemma:GST}
Let $u$ be a weight, and let $w$ be a function on $\N$. If $(\Delta- w) u = 0$ on $\N$, then
\[
T_u^{-1}\,(\Delta-w)\, T_u = \Delta_{u} \quad \mbox{ on }\; C_{c}(\N).
\]
\end{lemma}

\begin{proof}
We compute for $n \geq 1$,
\begin{multline*}
\Delta\,T_u\varphi(n) = \sum_{m \sim n} \big( u(n)\varphi(n) - u(m)\varphi(m) \big) \\
= (\Delta - w) u(n) \varphi(n) +    u(n)w(n)\varphi(n)
 + \sum_{m \sim n} u(m)\big( \varphi(n)-\varphi(m) \big) \\
= 0 + w(n)T_u\varphi(n) + \big(T_u \Delta_{u} \varphi \big)(n).\qedhere
\end{multline*}
\end{proof}
For a weight $u$ on $\N$, we define a quadratic
form $h_u$ on the space $C_c(\N)$, by
\[
h_u (\varphi) := \frac{1}{2} \sum_{n \in \N} \sum_{m \sim n} u(n)u(m) \big( \varphi(n)-\varphi(m) \big)^2 \qquad
\varphi \in C_c(\N),
\]
again with the ``Dirichlet boundary condition'' $ \varphi(0)=0 $ and $ u(0)=0 $.
For the constant weight $u=1$, we write $h:= h_1$. A direct calculation connects the form $h_u$ and $\Delta_u$ through the inner product on $\ell^2(\N,u^2)$.
\begin{lemma}[Green formula] \label{lemma:formU}
For a weight $u$ on $\N$,
we have
\[
\langle \Delta_{u} \varphi, \varphi \rangle_{u^2}=h_u (\varphi)\geq 0 \qquad \varphi \in C_c(\N).
\]
\end{lemma}

\begin{proof}
The proof follows by a direct calculation. Indeed,
\begin{multline*}
\langle \Delta_{u} \varphi,\varphi \rangle_{u^2} = \sum_{n \in \N} u^2(n)\,\Delta_{u}\varphi(n)\varphi(n) \\
= \!\!\sum_{n \in \N}\! \sum_{m \sim n}\!\! u(n)u(m) \big( \varphi(n)\!-\!\varphi(m) \big)\!^2
 \!\!+\!\! \sum_{n \in \N} \sum_{m \sim n} u(n)u(m)\big( \varphi(n)\!-\!\varphi(m) \big)\varphi(m) \\
= 2h_u(\varphi) - \langle \Delta_{u}\varphi, \varphi \rangle_{u^2}. \qedhere
\end{multline*}
\end{proof}
%%%%%%%%%%%%%%%%%%%
\begin{pro} \label{prop:AUX}
Let $u$ be a weight and suppose that $w \geq 0$ on $\N$. If
$(\Delta-w)u = 0$ on $\N$, then $h(\varphi) \geq \|\varphi\|_w^2$
for all $\varphi \in C_c(\N)$ with $\varphi(0)=0$.
\end{pro}
\begin{proof}
Using Lemma~\ref{lemma:unitary}, \ref{lemma:GST}, and \ref{lemma:formU} we obtain for all $\varphi \in C_c(\N)$,
\begin{multline*}
\langle (\Delta - w) \varphi,\varphi \rangle = \langle T_u^{-1}(\Delta- w) T_u T_u^{-1}\varphi,
T_u^{-1}\varphi \rangle_{u^2}\\
= \Big\langle \Delta_u \frac{\varphi}{u}, \frac{\varphi}{u} \Big\rangle_{u^2}=h_u \left(\frac{\varphi}{u}\right)\geq 0. \qedhere
\end{multline*}
\end{proof}
We are now in a position to prove the main theorem of this note.
\begin{proof}[Proof of Theorem~\ref{thm:MAIN}]
We define the function $w:\N \to \R$ as
\[
w(n):= \frac{\Delta n^{1/2}}{n^{1/2}} \qquad n \geq 1.
\]
 Applying the definition of the combinatorial Laplacian, we arrive
at
\begin{eqnarray*}
w(n) = 2- \Bigg( \Big( 1 + \frac{1}{n} \Big)^{1/2} + \Big(1-\frac{1}{n} \Big)^{1/2}  \Bigg)
\end{eqnarray*}
for $n \geq 1$. The Taylor expansion of the square root at $1$ gives for $n\ge2$
\begin{eqnarray*}
\Big( 1 \pm \frac{1}{n} \Big)^{1/2}\!\! =\! \sum_{k=0}^{\infty} \binom{1/2}{k} \Big( \frac{\pm 1}{n}\Big)^k
\!\!=\! 1 \pm \frac{1}{2n} - \frac{1}{8n^2} \pm \frac{1}{16n^3} - \frac{5}{128n^4} \pm \dots,
\end{eqnarray*}
and thus, it leads to the claimed series representation for $w(n)$, i.e.\@
\[
w(n) = -\sum_{k=1}^{\infty} \binom{1/2}{2k} \frac{2}{n^{2k}} = \sum_{k=1}^{\infty} \binom{4k}{2k}
\frac{1}{(4k-1)2^{4k-1}} \frac{1}{n^{2k}} \qquad n \ge2.
\]
Note that $w(1) = 2-\sqrt{2}$. So, $w >w_H> 0$ on $\N$.
By construction, we have $(\Delta-w)\,u = 0$ on $\N$ for the weight $u(n):=n^{1/2}$, $ n\ge1 $.
The validity of the Hardy inequality now follows from Proposition~\ref{prop:AUX}.
\end{proof}
 \begin{center}{\bf Acknowledgments} \end{center}
M.~K. acknowledges the financial support of the German Science Foundation. Y.~P. and F.~P. acknowledge the support of the Israel Science Foundation (grants No. 970/15) founded by the Israel Academy of Sciences and Humanities. Furthermore, F.~P. is grateful for support through a Technion Fine Fellowship.

\end{document}